\documentclass[11pt, a4paper, twoside]{article}
\usepackage{amsfonts}
\usepackage{mathrsfs}
\usepackage{amsmath, amsthm}
\usepackage{pgf,tikz}
\usepackage{amssymb}
\usepackage{fancyhdr}
\usepackage[colorlinks, linkcolor=black, anchorcolor=black, citecolor=black]{hyperref}

\numberwithin{equation}{section}

\allowdisplaybreaks

\begin{document}

\fancyhf{}

\fancyhead[OR]{\thepage}

\renewcommand{\headrulewidth}{0pt}
\renewcommand{\thefootnote}{\fnsymbol {footnote}}

\theoremstyle{plain} 
\newtheorem{thm}{\indent\sc Theorem}[section] 
\newtheorem{lem}[thm]{\indent\sc Lemma}
\newtheorem{cor}[thm]{\indent\sc Corollary}
\newtheorem{prop}[thm]{\indent\sc Proposition}
\newtheorem{claim}[thm]{\indent\sc Claim}
\theoremstyle{definition} 
\newtheorem{dfn}[thm]{\indent\sc Definition}
\newtheorem{rem}[thm]{\indent\sc Remark}
\newtheorem{ex}[thm]{\indent\sc Example}
\newtheorem{notation}[thm]{\indent\sc Notation}
\newtheorem{assertion}[thm]{\indent\sc Assertion}
%
%
\numberwithin{equation}{section}
\renewcommand{\proofname}{\indent\sc Proof.} 
\def\C{\mathbb{C}}
\def\R{\mathbb{R}}
\def\Rn{{\mathbb{R}^n}}
\def\M{\mathbb{M}}
\def\N{\mathbb{N}}
\def\Q{{\mathbb{Q}}}
\def\Z{\mathbb{Z}}
\def\F{\mathcal{F}}
\def\L{\mathcal{L}}
\def\S{\mathcal{S}}
\def\supp{\operatorname{supp}}
\def\essi{\operatornamewithlimits{ess\,inf}}
\def\esss{\operatornamewithlimits{ess\,sup}}
\def\dlim{\displaystyle\lim}

\fancyhf{}

\fancyhead[EC]{W. LI, H. Wang}

\fancyhead[EL]{\thepage}

\fancyhead[OC]{Non-tangential Convergence for Schr\"{o}dinger Operators}

\fancyhead[OR]{\thepage}

\renewcommand{\headrulewidth}{0pt}
\renewcommand{\thefootnote}{\fnsymbol {footnote}}

\title{\textbf{A Note on Non-tangential Convergence for Schr\"{o}dinger Operators}
\footnotetext {This work is supported by the Natural Science Foundation of China (No.11871452); Natural Science Foundation of China (No.11701452)
China Postdoctoral Science Foundation (No.2017M613193);  Natural Science Basic Research Plan in Shaanxi Province of China (No.2017JQ1009). }
\footnotetext {{}{2000 \emph{Mathematics Subject
 Classification}: 42B20, 42B25, 35S10.}}
\footnotetext {{}\emph{Key words and phrases}: Schr\"{o}dinger operator, Non-tangential convergence. } } \setcounter{footnote}{0}
\author{
Wenjuan Li, Huiju Wang, and Dunyan Yan}

\date{}
\maketitle

\begin{abstract}
The goal of this note is to establish non-tangential convergence results  for Schr\"{o}dinger operators along restricted curves. We consider the relationship between the dimension of this kind of approach region and  the regularity for the initial data which implies convergence. As a consequence, we obtain a upper bound for $p$ such that the Schr\"{o}dinger maximal function is bounded from $H^{s}(\mathbb{R}^{n})$ to $L^{p}(\mathbb{R}^{n})$ for any $s > \frac{n}{2(n+1)}$.
\end{abstract}

\section{Introduction}
The solution to the Schr\"{o}dinger equation
\begin{equation}\label{Eq11}
i{u_t} - \Delta u = 0, \hspace{0.5cm}(x,t) \in {\mathbb{R}^n} \times \mathbb{R}^{+},
\end{equation}
with initial datum $u\left( {x,0} \right) = f,$ is formally written as
\[{e^{it\Delta }}f\left( x \right): = \int_{{\mathbb{R}^n}} {{e^{i\left( {x \cdot \xi  + t{{\left| \xi  \right|}^2}} \right)}}\widehat{f}} \left( \xi  \right)d\xi .\]
The problem about finding optimal $s$ for which
\begin{equation}\label{Eq12}
\mathop {\lim }\limits_{t \to 0^{+}} {e^{it\Delta }}f\left( x \right) = f(x) \hspace{0.2cm} a.e.
\end{equation}
whenever $f \in {H^s}\left( {{\mathbb{R}^n}} \right),$ was first considered by Carleson \cite{C}, and extensively studied by Sj\"{o}lin \cite{S} and Vega \cite{V}, who proved independently the convergence for $s > 1/2$ in all dimensions. Dahlberg-Kenig \cite{DK} showed that the convergence does not hold for $s < 1/4$ in any dimension. In 2016, Bourgain \cite{B} gave conterexample showing that convergence can fail if $s<\frac{n}{2(n+1)}$. Very recently, Du-Guth-Li \cite{DGL} and Du-Zhang \cite{DZ} obtained the sharp results by the polynomial partitioning and decoupling method.

The natural generalization of the pointwise convergence problem is to ask a.e. convergence along a wider approach region instead of vertical lines. One of such problems is to consider non-tangential convergence to the initial data, it is natural to expect that more regularity on the initial data is necessary to guarantee a.e. existence of the non-tangential limit. It was shown by Sj\"{o}lin-Sj\"{o}gren \cite{SS} that non-tangential convergence fails for $s \le n/2$, i.e., there exists a function $f \in H^{\frac{n}{2}}(\mathbb{R}^{n})$ such that
 \begin{equation*}
 \limsup_{\substack{(y,t) \rightarrow (x,0) \\ |x-y| < \gamma(t), t >0}}|e^{it\Delta}f(y)| =\infty,
 \end{equation*}
 for all $x \in \mathbb{R}^{n}$, where $\gamma$ is strictly increasing and $\gamma(0)=0$.  Cho-Lee-Vargas \cite{CLV} raised a question about how the size or dimension of the approach region and the regularity which implies convergence are related.

In \cite{CLV}, this question is  considered in the one dimensional case. More concretely, let $\Gamma_x=\{x+t\theta:t\in [-1,1] $ and $\theta\in \Theta\}$, where $\Theta$ is a given compact set in $\mathbb{R}^{1}$. In \cite{CLV}, they proved that the corresponding non-tangential convergence result holds for $s>\frac{\beta(\Theta)+1}{4}$, here $\beta(\Theta)$ denotes the upper Minkowski dimension of $\Theta$. This result in \cite{CLV} was established by the $TT^{\star}$ method and a time localizing lemma. Recently, by getting around the key localizing lemma in \cite{CLV}, Shiraki \cite{Shi} generalized this result to a wider class of equations which includes the fractional Schr\"{o}dinger equation.

However, the above question remains open in higher dimensional case until recently. In this article, we consider the non-tangential convergence problem along the approach region in $\mathbb{R}^{n}$ given by
\[\Gamma_x=\{\gamma(x,t,\theta):t\in [0,1], \theta\in \Theta\},\]
where $\Theta$ is a given compact set in $\mathbb{R}^{n}$. $\gamma$ is a map from $\mathbb{R}^{n} \times [0,1] \times \Theta$ to $\mathbb{R}^{n}$,  which satisfies $\gamma(x,0,\theta) =x$ for all $x \in \mathbb{R}^{n}$, $\theta \in \Theta$, and the following (C1)-(C3) hold:\\
(C1) For fixed $t \in [0,1]$, $\theta \in \Theta$, $\gamma$ has at least $C^{1}$ regularity in $x$, and there exists a constant $C_{1} \ge 1$
such that for each $x, x^{\prime}  \in \mathbb{R}^{n}$, $\theta \in \Theta$, $t \in [0,1]$,
\begin{equation}\label{Eq1.3}
 C_{1}^{-1}|x- x^{\prime}| \le |\gamma(x,t, \theta)-\gamma(x^{\prime},t, \theta)| \le C_{1}|x- x^{\prime}|;
\end{equation}
(C2) There exists  a constant $C_{2} >0$ such that for each $x \in \mathbb{R}^{n}$, $\theta \in \Theta$, $t,t^{\prime} \in [0,1]$,
\begin{equation}\label{Eq1.4}
|\gamma(x,t, \theta)-\gamma(x,t^{\prime}, \theta)| \le C_{2}|t- t^{\prime}|;
\end{equation}
(C3) There exists a constant $C_{3} >0$ such that for each $x \in \mathbb{R}^{n}$, $t \in [0,1]$, $\theta, \theta^{\prime} \in \Theta$,
\begin{equation}\label{Eq1.5}
|\gamma(x,t, \theta)-\gamma(x,t, \theta^{\prime})| \le C_{3}|\theta- \theta^{\prime}|.
\end{equation}
We consider the relationship between the dimension of $\Theta$ and the optimal $s$ for which
 \begin{equation}
 \lim_{\substack{(y,t) \rightarrow (x,0) \\ y\in \Gamma_x}}e^{it\Delta}f(y) = f(x) \hspace{0.2cm} a.e.
 \end{equation}
whenever $f \in {H^s}\left( {{\mathbb{R}^n}} \right)$.

We first give two examples for $\Gamma_{x}$. It is not hard to check that all the conditions mentioned above can be satisfied if we take (E1): $\gamma (x,t,\theta) = x + t\theta$, $\Theta$ is a compact subset of the unit ball in $\mathbb{R}^{n}$.
When $n=1,$ this is just the problem considered in \cite{CLV}. Another example is (E2): $\gamma (x,t,\theta) = x + t^{\theta}$, $t^{\theta} =(t^{\theta_{1}}, t^{\theta_{2}}, \cdot \cdot \cdot, t^{\theta_{n}} )$, $\theta=(\theta_1,\theta_2,\cdots,\theta_n)$, $\Theta$ is a compact subset in the first quadrant away from the axis of $\mathbb{R}^{n}$. For this example, it is worth to mention that when  $\theta$ is fixed,  Lee-Rogers \cite{LR} have obtained that the convergence along the curve $(\gamma_{\theta}(x,t),t)$ is equivalent to the convergence along the vertical line.

\begin{center}
\includegraphics[height=2cm]{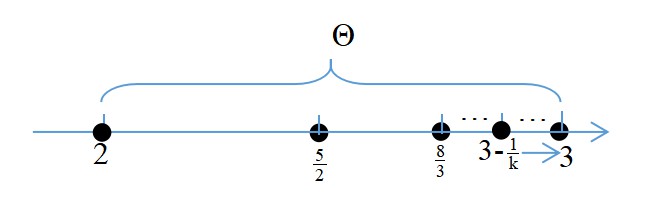}
\end{center}
\begin{center}
Figure 1. $\Theta=\{2,5/2,\cdots,3-1/k,\cdots,3:k=1,2,\cdots\}$.
\end{center}

\begin{center}
\includegraphics[height=12cm]{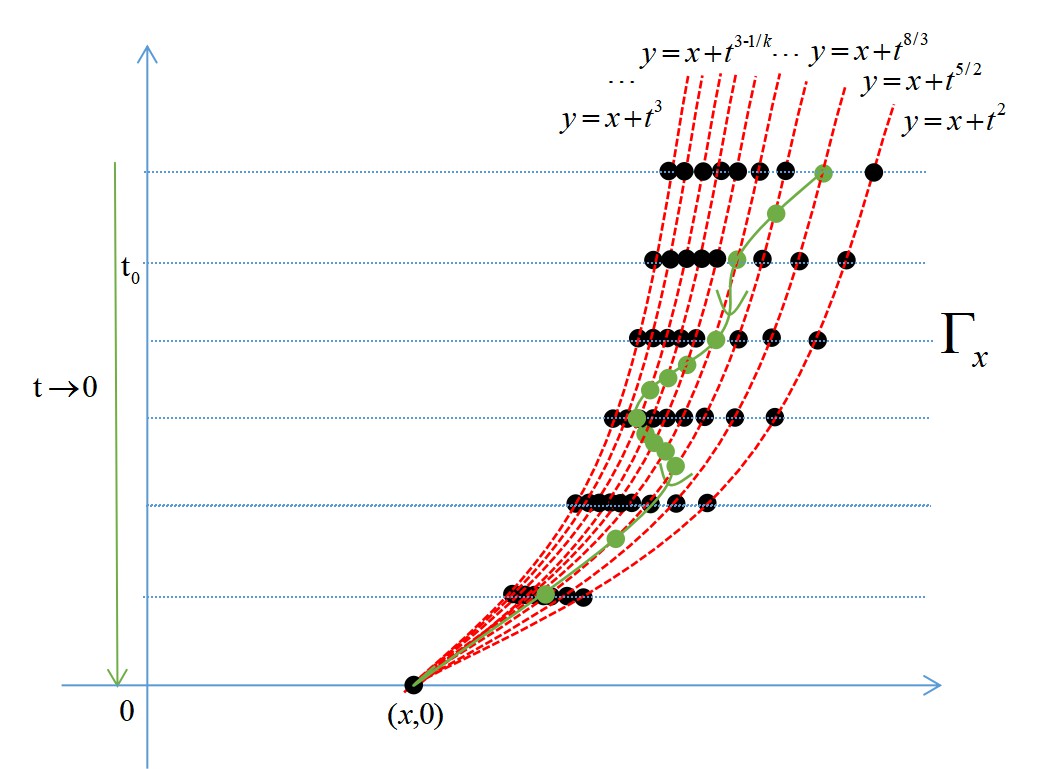}
\end{center}
\begin{center}
Figure 2. $\Gamma_x$ is consist of all black points which lie on the line $y=x+t^{\theta}$, $\theta\in \Theta$, $t\in [0,1/2]$. For every layer $t=t_0$, there are countable black points corresponding to $\Theta$. We try to seek the optimal $s$ for $e^{it\Delta}f(y)\rightarrow f(x)$ along different green-path whose points from $\Gamma_x$ whenever $f\in H^s$.
\end{center}

In order to characterize the size of $\Theta$, we introduce the so called logarithmic density or upper Minkowski dimension of $\Theta$,
\[\beta(\Theta)= \limsup_{\delta \rightarrow 0^{+}} \frac{logN(\delta)}{-log\delta},\]
where $N(\delta)$ is the minimum number of closed balls of diameter $\delta$ to cover $\Theta$. It is not hard to see that when $\Theta$ is a single point, $\beta(\Theta) =0$; when $\Theta$ is a compact subset of $\mathbb{R}^{n}$ with positive Lebesgue measure, $\beta(\Theta) =n$.

By standard arguments, in order to obtain the convergence result, it is sufficient to establish the bounded estimates for the maximal operator defined by
\[\mathop{sup}_{(t,\theta) \in (0,1) \times \Theta} |e^{it\Delta}f(\gamma(x,t,\theta))|.\]
Our main results are as follows. Firstly, we show the maximal operator estimate in the two dimensional case.

\begin{thm}\label{theorem1.1}
When $n=2$, given $B(x_{0},R) \subset \mathbb{R}^{2}$, $R \lesssim 1$, then for any $s> \frac{\beta(\Theta)+1}{3}$,
\begin{equation}\label{Eq1.7+}
\biggl\|\mathop{sup}_{(t,\theta) \in (0,1) \times \Theta} |e^{it\Delta}f(\gamma(x,t,\theta))|\biggl\|_{L^{3}(B(x_{0},R))}  \leq C\|f\|_{H^s(\mathbb{R}^2)},
\end{equation}
whenever $f\in H^s(\mathbb{R}^2)$, where the constant $C$ depends on $s$, $C_{1}$, $C_{2}$, $C_{3}$, and the chosen of $ B(x_{0}, R)$, but does not depend on $f$.
\end{thm}

Then we obtain the following non-tangential convergence result.

\begin{thm}\label{theorem1.2}
When $n=2$, if $s> \frac{\beta(\Theta)+1}{3}$, then
\begin{equation}\label{Eq1.8+}
 \lim_{\substack{(y,t) \rightarrow (x,0) \\ y\in\Gamma_x}}e^{it\Delta}f(y) = f(x) \hspace{0.2cm} a.e.
\end{equation}
whenever $f\in H^s(\mathbb{R}^2)$.
\end{thm}

We notice that the  convergence result obtained in Theorem \ref{theorem1.2} is sharp when $\beta(\Theta) = 0$ (\cite{DGL} and \cite{B}) or $\beta(\Theta) = 2$ (\cite{SS}). It is quite interesting to seek whether (\ref{Eq1.8+}) is sharp  when $0 < \beta(\Theta) <2$.

We briefly sketch the proof of Theorem \ref{theorem1.1}, and leave the details to Section 2. We decompose $\Theta$ into small subsets $\Theta= \cup_{k}\Theta_{k}$ with bounded overlap, for each $\Theta_{k}$, the size is small enough such that our problem can be reduced to estimate the maximal function for Schr\"{o}dinger operator along certain curves, i.e. the maximal operator defined by
\begin{equation}\label{Eq1.7}
\mathop{sup}_{t \in (0,1)} |e^{it\Delta}f(\gamma(x,t,\theta_{k}^{0})|
\end{equation}
for some $\theta_{k}^{0} \in \Theta_{k}$.
The number of $\Theta_{k}$ is determined by $\beta(\Theta)$. Finally, in order to get the bounded estimate for maximal function defined by (\ref{Eq1.7}), we still need the following theorem.
\begin{thm}\label{lemma3.2}(\cite{DGL})
 For any $s>1/3$, the following bound holds: for any function $f\in H^s(\mathbb{R}^2)$,
\begin{equation*}
\biggl\|\mathop{sup}_{0<t<1}|e^{it\Delta}f(x)|\biggl\|_{L^3(B(0,1))} \leq C_s\|f\|_{H^s(\mathbb{R}^2)}.
\end{equation*}
\end{thm}

The idea to establish the bounded estimate for maximal function define by (\ref{Eq1.7}) using Theorem \ref{lemma3.2} comes from the method adopted by Lee-Rogers \cite{LR} to show equivalence between convergence result for Schr\"{o}dinger operators along smooth curves and vertical lines. However, we should be more careful since we need an estimate uniformly in $k$. In our case, this can be realized sine $\Theta$ is compact.

The above argument can also be applied to obtain the  corresponding convergence problem in general dimensions. So we have the following result.

\begin{thm}\label{theorem1.4}
For general positive integer $n$, if there exists $p >1$ such that for any $s> \frac{n}{2(n+1)}$,
 \begin{equation}\label{Eq1.8}
\biggl\|\mathop{sup}_{0<t<1}|e^{it\Delta}f(x)|\biggl\|_{L^p(B(0,1))} \leq C_s\|f\|_{H^s(\mathbb{R}^n)}
\end{equation}
whenever $f \in H^{s}(\mathbb{R}^{n})$, then given $B(x_{0},R) \subset \mathbb{R}^{n}$, $R \le 1$, for any $s> \frac{\beta(\Theta)}{p} + \frac{n}{2(n+1)}$,
\begin{equation}
\biggl\|\mathop{sup}_{(t,\theta) \in (0,1) \times \Theta} |e^{it\Delta}f(\gamma(x,t,\theta))|\biggl\|_{L^{p}(B(x_{0},R))}  \leq C\|f\|_{H^s(\mathbb{R}^n)},
\end{equation}
whenever  $f\in H^s(\mathbb{R}^n)$, where the constant $C$ depends on $s$, $C_{1}$, $C_{2}$, $C_{3}$, and the chosen of $ B(x_{0}, R)$, but does not depend on $f$.
\end{thm}

Theorem \ref{theorem1.4} implies that if (\ref{Eq1.8}) holds for some $p >1$ whenever $f \in H^{s}(\mathbb{R}^{n})$ for any $s> \frac{n}{2(n+1)}$, then for any $s> \frac{\beta(\Theta)}{p} + \frac{n}{2(n+1)}$,
\begin{equation}
 \lim_{\substack{(y,t) \rightarrow (x,0) \\ y\in \Gamma_x}}e^{it\Delta}f(y) = f(x) \hspace{0.2cm} a.e.
 \end{equation}
whenever $f \in H^{s}(\mathbb{R}^{n})$. Then it comes to the question about what is the optimal $p$ for (\ref{Eq1.8}) to hold for any $s> \frac{n}{2(n+1)}$. This question is still open to our best knowledge, but combining with the counterexample given by Sj\"{o}lin-Sj\"{o}gren \cite{SS}, we get a upper bound for $p$.

\begin{thm}\label{theorem1.5}
For general positive integer $n$, if there exists $p >1$ such that for any $s> \frac{n}{2(n+1)}$, (\ref{Eq1.8}) holds
whenever $f \in H^{s}(\mathbb{R}^{n})$, then $p \le \frac{2(n+1)}{n}$.
\end{thm}

The upper bound given by Theorem \ref{theorem1.5} is sharp for $n=1$ (\cite{C}) and $n=2$ (\cite{DGL}), but we do not know if it is also sharp for $n \ge 3$. By parabolic rescaling and time localizing lemma, inequality (\ref{Eq1.8}) is equivalent to
 \begin{equation}\label{Eq1.11}
\biggl\|\mathop{sup}_{0<t<\lambda}|e^{it\Delta}f(x)|\biggl\|_{L^p(B(0,\lambda))} \leq C_\epsilon \lambda^{n(\frac{1}{p}-\frac{n}{2(n+1)}) + \epsilon} \|f\|_{L^2(\mathbb{R}^n)}
\end{equation}
whenever supp$\hat{f} \subset \{\xi \in \mathbb{R}^{n}: |\xi| \sim 1\}$, where $\lambda \gg 1$. The range of $p$ has been discussed in Du-Kim-Wang-Zhang \cite{DKWZ}, but the optimal range of $p$ is still unknown.

\textbf{Conventions}: Throughout this article, we shall use the well known notation $A\gg B$, which means if there is a sufficiently large constant $G$, which does not depend on the relevant parameters arising in the context in which
the quantities $A$ and $B$ appear, such that $ A\geq GB$. We write $A\sim B$, and mean that $A$ and $B$ are comparable. By
$A\lesssim B$ we mean that $A \le CB $ for some constant $C$ independent of the parameters related to  $A$ and $B$.  Given $\mathbb{R}^{n}$, we write $B(0,1)$ instead of the unit ball $B^{n}(0,1)$ in $\mathbb{R}^{n}$ centered at the origin for short, and the same notation is valid for $B(x_{0},R)$.

\section{Proof of the main theorems for $n=2$}
\textbf{Proof of Theorem \ref{theorem1.1}.} In order to prove Theorem \ref{theorem1.1}, using Littlewood-Paley decomposition, we only need to show that for $f$ with supp$\hat{f} \subset \{\xi \in \mathbb{R}^{2}: |\xi| \sim \lambda\}$, $\lambda \gg 1$,
 \begin{equation}\label{Eq2.1}
\biggl\|\mathop{sup}_{(t,\theta) \in (0,1) \times \Theta} |e^{it\Delta}f(\gamma(x,t,\theta))|\biggl\|_{L^{3}(B(x_{0},R))} \le C\lambda^{s_{0}+\epsilon} \|f\|_{L^{2}}, \:\ \forall \epsilon >0,
\end{equation}
where $s_{0}=\frac{\beta(\Theta)+1}{3}$.

 We decompose $\Theta$ into subsets $\Theta= \cup_{k}\Theta_{k}$ with bounded overlap, where each $\Theta_{k}$ is contained in a closed ball with diameter $\lambda^{-1}$. Then we have
\begin{equation}\label{Eq2.2}
1 \le k \le \lambda^{\beta(\Theta) + \epsilon}.
\end{equation}
 We claim that for each $k$,
 \begin{equation}\label{Eq2.3}
\biggl\|\mathop{sup}_{(t,\theta) \in (0,1) \times \Theta_{k}} |e^{it\Delta}f(\gamma(x,t,\theta))|\biggl\|_{L^{3}(B(x_{0},R))} \le C\lambda^{\frac{1}{3}+\frac{2\epsilon}{3}} \|f\|_{L^{2}}.
\end{equation}
Then inequality (\ref{Eq2.1}) follows from (\ref{Eq2.2}) and (\ref{Eq2.3}).  More concretely, we have
\begin{align}
\biggl\|\mathop{sup}_{(t,\theta) \in (0,1) \times \Theta} |e^{it\Delta}f(\gamma(x,t,\theta))|\biggl\|_{L^{3}(B(x_{0},R))}
&\le \biggl(\sum_{k} \biggl\|\mathop{sup}_{(t,\theta) \in (0,1) \times \Theta_{k}} |e^{it\Delta}f(\gamma(x,t,\theta))|\biggl\|_{L^{3}(B(x_{0},R))}^{3}  \biggl)^{1/3} \nonumber\\
&\le C\biggl(\sum_{k} \lambda^{1+2\epsilon} \|f\|_{L^{2}}^{3} \biggl)^{1/3} \nonumber\\
&\le  C\lambda^{\frac{\beta{\Theta}+1}{3}+\epsilon} \|f\|_{L^{2}},
\end{align}
which implies inequality (\ref{Eq2.1}).

Now we are left to prove inequality (\ref{Eq2.3}). For this goal, we first show the following Lemma \ref{lemma2.1}. The original idea comes from Lemma 2.2 in \cite{LR}.

\begin{lem}\label{lemma2.1}
Assume that $g$ is a Schwartz function whose Fourier transform is supported in  $\{\xi \in \mathbb{R}^{n}: |\xi| \sim \lambda\}$.  If
\[ |\theta - \theta^{\prime}| \le \lambda^{-1},\]
then for each $x \in B(x_{0},R)$ and $t \in (0, 1)$,
\begin{align}\label{Eq2.5+}
|e^{it\Delta}g(\gamma(x,t,\theta))| \le \sum_{\mathfrak{l} \in \mathbb{Z}^{n}}{\frac{C}{(1+|\mathfrak{l}|)^{n+1}}\biggl|\int_{\mathbb{R}^{n}}{e^{i[\gamma(x,t,\theta^{\prime})+ \frac{\mathfrak{l}}{\lambda}]\cdot\xi+it|\xi|^{2}}\hat{g}}(\xi)d\xi \biggl|},
\end{align}
where the constant $C$ depends on $n$ and $C_{3}$ in inequality (\ref{Eq1.5}).
\end{lem}

\begin{proof}
We introduce a cut-off function $\phi$ which is smooth and equal to $1$ on $B(0,2)$ and supported on $(-\pi, \pi)^{n}$. After scaling we have
\begin{align}
e^{it\Delta}g(\gamma(x,t,\theta))&= \lambda^{n} \int_{\mathbb{R}^{n}}{e^{i\lambda \gamma(x,t,\theta)\cdot \eta + it|\lambda \eta|^{2}} \phi(\eta)\hat{g}(\lambda \eta) d\eta} \nonumber\\
&= \lambda^{n} \int_{\mathbb{R}^{n}}{e^{i\lambda \gamma(x,t, \theta)\cdot \eta -i\lambda \gamma(x,t,\theta^{\prime}) \cdot \eta  +i\lambda \gamma(x,t,\theta^{\prime}) \cdot \eta + it|\lambda \eta|^{2}} \phi(\eta)\hat{g}(\lambda \eta) d\eta}.
\end{align}
Since it follows by inequality (\ref{Eq1.5}),
\[\lambda|\gamma(x,t,\theta)-\gamma(x,t,\theta^{\prime})| \le C_{3},\]
then by Fourier expansion,
\[\phi(\eta)e^{i\lambda[\gamma(x,t,\theta)-\gamma(x,t,\theta^{\prime})] \cdot \eta} = \sum_{\mathfrak{l} \in \mathbb{Z}^{n}}{c_{\mathfrak{l}}(x,t,\theta, \theta^{\prime})e^{i\mathfrak{l}\cdot \eta}},\]
where
\[|c_{\mathfrak{l}}(x,t,\theta, \theta^{\prime})| \le \frac{C}{(1+|\mathfrak{l}|)^{n+1}}\]
uniformly for each $\mathfrak{l} \in \mathbb{Z}^{n}$, $x \in B(x_{0},R)$ and $t \in (0, 1)$. Then we have
\begin{align}
|e^{it\Delta}g(\gamma(x,t,\theta))| &\le \sum_{\mathfrak{l} \in \mathbb{Z}^{n}}{\frac{C\lambda^{n}}{(1+|\mathfrak{l}|)^{n+1}}\biggl|\int_{\mathbb{R}^{n}}{e^{i \mathfrak{l} \cdot \eta  +i\lambda \gamma(x,t,\theta^{\prime}) \cdot \eta + it|\lambda \eta|^{2}} \hat{g}(\lambda \eta) d\eta} \biggl|} \nonumber\\
&= \sum_{\mathfrak{l} \in \mathbb{Z}^{n}}{\frac{C}{(1+|\mathfrak{l}|)^{n+1}}\biggl|\int_{\mathbb{R}^{n}}{e^{i \frac{\mathfrak{l}}{\lambda} \cdot \xi  +i \gamma(x,t,\theta^{\prime}) \cdot \xi + it|\xi|^{2}} \hat{g}(\xi) d\xi} \biggl|}, \nonumber
\end{align}
then we arrive at (\ref{Eq2.5+}).
\end{proof}

By the similar argument, we can prove the following lemma.

\begin{lem}\label{lemma2.2}
Assume that $g$ is a Schwartz function whose Fourier transform is supported in  $\{\xi \in \mathbb{R}^{n}: |\xi| \sim \lambda\}$. If
\[ |t - t^{\prime}| \le \lambda^{-1},\]
then for each $x \in B(x_{0},R)$ and $\theta \in \Theta$,
\begin{align}
|e^{it\Delta}g(\gamma(x,t,\theta))| \le \sum_{\mathfrak{l} \in \mathbb{Z}^{n}}{\frac{C}{(1+|\mathfrak{l}|)^{n+1}}\biggl|\int_{\mathbb{R}^{n}}{e^{i[\gamma(x,t^{\prime},\theta)+ \frac{\mathfrak{l}}{\lambda}]\cdot\xi+it|\xi|^{2}}\hat{g}}(\xi)d\xi \biggl|}.
\end{align}
where the constant $C$ depends on $n$ and $C_{2}$ in inequality (\ref{Eq1.4}).
\end{lem}

We now prove inequality (\ref{Eq2.1}). For fixed $k$, by the construction of $\Theta_{k}$, there is a $\theta_{k}^{0} \in \Theta_{k}$ such that
\[ |\theta - \theta_{k}^{0}| \le \lambda^{-1}\]
holds for each $\theta \in \Theta_{k}$. Then according to Lemma \ref{lemma2.1},
for each $x \in B(x_{0},R)$, $t \in (0, 1)$ and $\theta \in \Theta_{k}$, we have
\begin{align}\label{Eq2.8}
|e^{it\Delta}f(\gamma(x,t,\theta))| &\le \sum_{\mathfrak{l} \in \mathbb{Z}^{2}}{\frac{C}{(1+|\mathfrak{l}|)^{3}}\biggl|\int_{\mathbb{R}^{2}}{e^{i \gamma(x,t,\theta_{k}^{0})  \cdot\xi+it|\xi|^{2}} e^{i\frac{\mathfrak{l}}{\lambda} \cdot \xi}\hat{f}}(\xi)d\xi \biggl|} \nonumber\\
&= \sum_{\mathfrak{l} \in \mathbb{Z}^{2}}{\frac{C}{(1+|\mathfrak{l}|)^{3}}\biggl|\int_{\mathbb{R}^{2}}{e^{i \gamma(x,t,\theta_{k}^{0})  \cdot\xi+it|\xi|^{2}} \hat{f_{\lambda}^{\mathfrak{l}}}}(\xi)d\xi \biggl|} \nonumber\\
&= \sum_{\mathfrak{l} \in \mathbb{Z}^{2}}{\frac{C}{(1+|\mathfrak{l}|)^{3}}\biggl| e^{it\Delta} f_{\lambda}^{\mathfrak{l}} (\gamma(x,t,\theta_{k}^{0}))  \biggl|},
\end{align}
where
\[ \hat{f_{\lambda}^{\mathfrak{l}}} (\xi)=e^{i\frac{\mathfrak{l}}{\lambda} \cdot \xi}\hat{f}(\xi). \]
It follows that
\begin{align}
&\biggl\|\mathop{sup}_{(t,\theta) \in (0,1) \times \Theta_{k}} |e^{it\Delta}f(\gamma(x,t,\theta))|\biggl\|_{L^{3}(B(x_{0},R))} \nonumber\\
&\le \sum_{\mathfrak{l} \in \mathbb{Z}^{2}} \frac{C}{(1+|\mathfrak{l}|)^{3}} \biggl\|\mathop{sup}_{(t,\theta) \in (0,1) \times \Theta_{k}} |e^{it\Delta}f_{\lambda}^{\mathfrak{l}}(\gamma(x,t,\theta_{k}^{0}))|\biggl\|_{L^{3}(B(x_{0},R))} \nonumber\\
&= \sum_{\mathfrak{l} \in \mathbb{Z}^{2}} \frac{C}{(1+|\mathfrak{l}|)^{3}} \biggl\|\mathop{sup}_{t \in (0,1)} |e^{it\Delta}f_{\lambda}^{\mathfrak{l}}(\gamma(x,t,\theta_{k}^{0}))|\biggl\|_{L^{3}(B(x_{0},R))} \nonumber\\
&\le \sum_{\mathfrak{l} \in \mathbb{Z}^{2}} \frac{C}{(1+|\mathfrak{l}|)^{3}} \lambda^{\frac{1}{3}+ \frac{2\epsilon}{3}} \|f_{\lambda}^{\mathfrak{l}}\|_{L^{2}} \nonumber\\
&\le \lambda^{\frac{1}{3}+ \frac{2\epsilon}{3}} \|f\|_{L^{2}},
\end{align}
provided that we have proved the following lemma.

\begin{lem}\label{lemma2.3}
Assume that $g$ is a Schwartz function whose Fourier transform is supported in the annulus $\{\xi \in \mathbb{R}^{2}: |\xi| \sim \lambda\}$. Then for each $k$,
\begin{align}\label{Eq2.1n}
\biggl\|\mathop{sup}_{t \in (0,1)} |e^{it\Delta}g(\gamma(x,t,\theta_{k}^{0}))|\biggl\|_{L^{3}(B(x_{0},R))} \le C \lambda^{\frac{1}{3}+ \frac{2\epsilon}{3}} \|g\|_{L^{2}},
\end{align}
where the constant $C$ is independent of $k$.
\end{lem}

Now let's turn to prove Lemma \ref{lemma2.3}. The following theorem is required.

\begin{thm}\label{theorem2.4}(\cite{LR})
Let $\rho: \mathbb{R}^{n+1} \rightarrow \mathbb{R}^{n}$, $q, r \in [2, +\infty]$, $\lambda \ge 1$, supp $\nu \subset [-2,2]$, $\lambda \ge \|1\|_{L_{\mu}^{q}L_{\nu}^{r}}^{1/n}$, and suppose that
\[\mathop{sup}_{x \in supp(\nu), t \in supp(\nu)}|\rho(x,t)| \le M,\]
where $M >1$. Suppose that for a collection of boundedly overlapping intervals $I$ of length $\lambda^{-1}$, there exists a $C_{0} >1$ such that
\begin{equation*}
\|e^{it\Delta}f(\rho(x,t))\|_{L_{\mu}^{q}L_{\nu}^{r}(I)} \leq C_0\|f\|_{L^2(\mathbb{R}^n)},
\end{equation*}
whenever $\hat{f}$ is supported in  $\{\xi \in \mathbb{R}^{n}: |\xi| \sim \lambda\}$. Then there is a constant $C_{n} >1$ such that
\begin{equation*}
\|e^{it\Delta}f(\rho(x,t))\|_{L_{\mu}^{q}L_{\nu}^{r}(\cup I)} \leq C_{n} M^{1/2} C_0\|f\|_{L^2(\mathbb{R}^n)}
\end{equation*}
whenever $\hat{f}$ is supported in  $\{\xi \in \mathbb{R}^{n}: |\xi| \sim \lambda\}$.
\end{thm}

Notice that in our case, for each $k$, we have
\[\mathop{sup}_{(x,t) \in B(x_{0},R) \times (0,1)}|\gamma(x,t,\theta_{k}^{0})| \le \mathop{sup}_{(x,t,\theta) \in B(x_{0},R) \times (0,1) \times \Theta}|\gamma(x,t,\theta)|.\]
By inequality (\ref{Eq1.4}), for each $(x,t,\theta) \in B(x_{0},R) \times (0,1) \times \Theta$,
\[|\gamma(x,t,\theta) - \gamma(x,0,\theta)| \le C_{2},\]
then $|\gamma(x,t,\theta)|$ is uniformly bounded for $(x,t,\theta) \in B(x_{0},R) \times (0,1) \times \Theta$, and the upper bound is  determined by $C_{2}$ and the chosen of $B(x_{0},R)$, but independent of $k$.

Therefore, according to Theorem \ref{theorem2.4}, in order to prove Lemma \ref{lemma2.3}, we only need to show that for each interval $I \subset (0,1)$ of length $\lambda^{-1}$, and any function $g$ such that $\hat{g}$ is supported in  $\{\xi \in \mathbb{R}^{2}: |\xi| \sim \lambda\}$, we have
\begin{align}\label{Eq2.11}
\biggl\|\mathop{sup}_{t \in I} |e^{it\Delta}g(\gamma(x,t,\theta_{k}^{0}))|\biggl\|_{L^{3}(B(x_{0},R))} \le C \lambda^{\frac{1}{3}+ \frac{2\epsilon}{3}} \|g\|_{L^{2}}.
\end{align}
 Since $I$ is an interval of length $\lambda^{-1}$, there exists $t_{I}^{0} \in I$ such that for each $t \in I$,
 \[ |t - t_{I}^{0}| \le \lambda^{-1}.\]
 Then by Lemma \ref{lemma2.2},
for each $x \in B(x_{0},R)$, $t \in I$, we have
\begin{align}\label{Eq2.8}
|e^{it\Delta}g(\gamma(x,t,\theta_{k}^{0}))| &\le \sum_{\mathfrak{l} \in \mathbb{Z}^{2}}{\frac{C}{(1+|\mathfrak{l}|)^{3}}\biggl|\int_{\mathbb{R}^{2}}{e^{i \gamma(x,t_{I}^{0},\theta_{k}^{0})  \cdot\xi+it|\xi|^{2}} e^{i\frac{\mathfrak{l}}{\lambda} \cdot \xi}\hat{g}}(\xi)d\xi \biggl|} \nonumber\\
&= \sum_{\mathfrak{l} \in \mathbb{Z}^{2}}{\frac{C}{(1+|\mathfrak{l}|)^{3}}\biggl| e^{it\Delta} g_{\lambda}^{\mathfrak{l}} (\gamma(x,t_{I}^{0},\theta_{k}^{0}))  \biggl|},
\end{align}
where
\[ \hat{g_{\lambda}^{\mathfrak{l}}} (\xi)=e^{i\frac{\mathfrak{l}}{\lambda} \cdot \xi}\hat{g}(\xi). \]
It follows that
\begin{align}\label{Eq2.13}
\biggl\|\mathop{sup}_{t \in I} |e^{it\Delta}g(\gamma(x,t,\theta_{k}^{0}))|\biggl\|_{L^{3}(B(x_{0},R))}
\le \sum_{\mathfrak{l} \in \mathbb{Z}^{2}} \frac{C}{(1+|\mathfrak{l}|)^{3}} \biggl\|\mathop{sup}_{t \in I} |e^{it\Delta}g_{\lambda}^{\mathfrak{l}}(\gamma(x,t_{I}^{0},\theta_{k}^{0}))|\biggl\|_{L^{3}(B(x_{0},R))}.
\end{align}
For each $t_{I}^{0}$, $\theta_{k}^{0}$, $\gamma_{t_{I}^{0}, \theta_{k}^{0}}$ is at least $C^{1}$ from $\mathbb{R}^{2}$ to $\mathbb{R}^{2}$. By inequality (\ref{Eq1.3}), for each $x \in \mathbb{R}^{2}$,
\[C_{1}^{-1} \le |\nabla_{x}\gamma(x,t_{I}^{0},\theta_{k}^{0})| \le C_{1}.\]
By the same reason, for each $x \in B(x_{0},R)$,
\[|\gamma(x,t_{I}^{0},\theta_{k}^{0}) - \gamma(x_{0},t_{I}^{0},\theta_{k}^{0})| \le C_{1}R,\]
which implies $\gamma_{t_{I}^{0}, \theta_{k}^{0}}(B(x_{0},R)) \subset B( \gamma(x_{0},t_{I}^{0},\theta_{k}^{0}),C_{1}R)$. Therefore, changes of variables and Theorem \ref{lemma3.2} imply that
\begin{align}\label{Eq2.14}
\biggl\|\mathop{sup}_{t \in I} |e^{it\Delta}g_{\lambda}^{\mathfrak{l}}(\gamma(x,t_{I}^{0},\theta_{k}^{0}))|\biggl\|_{L^{3}(B(x_{0},R))} \le C \lambda^{\frac{1}{3}+ \frac{2\epsilon}{3}} \|g_{\lambda}^{\mathfrak{l}}\|_{L^{2}}.
\end{align}
Combining inequality (\ref{Eq2.13}) with inequality (\ref{Eq2.14}), we have
\begin{align}
&\biggl\|\mathop{sup}_{t \in I} |e^{it\Delta}g(\gamma(x,t,\theta_{k}^{0}))|\biggl\|_{L^{3}(B(x_{0},R))} \nonumber\\
&\le \sum_{\mathfrak{l} \in \mathbb{Z}^{2}} \frac{C}{(1+|\mathfrak{l}|)^{3}} \lambda^{\frac{1}{3}+ \frac{2\epsilon}{3}} \|g_{\lambda}^{\mathfrak{l}}\|_{L^{2}} \nonumber\\
&\le C\lambda^{\frac{1}{3}+ \frac{2\epsilon}{3}} \|g\|_{L^{2}}.
\end{align}
This completes the proof of Lemma \ref{lemma2.3}.

\textbf{Proof of Theorem \ref{theorem1.2}.} The proof of Theorem \ref{theorem1.2} is quite standard. We write the details for completeness. In fact, for any  $s> \frac{\beta(\Theta)+1}{3}$, $f \in H^{s}(\mathbb{R}^{2})$, fix $\lambda >0$,  choose $g \in C_{c}^{\infty}(\mathbb{R}^{2})$ such that
\[\|f-g\|_{H^{s}(\mathbb{R}^{2})} \le  \frac{\lambda \epsilon^{1/3}}{2C},\]
where the constant $C$ is the constant in inequality (\ref{Eq1.7+}), which follows
\begin{align}\label{Eq2.4n}
&\biggl|\biggl\{ x \in B(x_{0},R): \mathop{sup}_{(t,\theta) \in (0,1) \times \Theta} |e^{it\Delta}(f-g)(\gamma(x,t,\theta))| > \frac{\lambda}{2}\biggl\}\biggl|  \nonumber\\
&\le \frac{2^{3}}{\lambda^{3}} \biggl\|\mathop{sup}_{(t,\theta) \in (0,1) \times \Theta} |e^{it\Delta}(f-g)(\gamma(x,t,\theta))|\biggl\|_{L^{3}(B(x_{0},R))}^{3} \nonumber\\
&\le \frac{2^{3}C^{3}}{\lambda^{3}}\|f-g\|_{H^{s}(\mathbb{R}^{2})}^{3} \nonumber\\
&\le \epsilon.
\end{align}
Moreover,
\begin{align}\label{Eq2.5n}
 \lim_{\substack{(y,t) \rightarrow (x,0) \\ y\in \Gamma_x}} e^{it\Delta}g(y) =g(x)
 \end{align}
uniformly for $x \in B(x_{0},R)$. Indeed, for each $x \in B(x_{0},R)$,
\begin{align}\label{Eq2.6n}
\limsup_{\substack{(y,t) \rightarrow (x,0) \\ y\in \Gamma_x}} |e^{it\Delta}g(y) - g(x)| &\le \limsup_{\substack{(y,t) \rightarrow (x,0) \\ y\in \Gamma_x}} |e^{it\Delta}g(y) - e^{it\Delta}g(x)| + \limsup_{\substack{(y,t) \rightarrow (x,0) \\ y\in \Gamma_x}} |e^{it\Delta}g(x) - g(x)| \nonumber\\
&= \limsup_{\substack{(y,t) \rightarrow (x,0) \\ y\in \Gamma_x}} |e^{it\Delta}g(y) - e^{it\Delta}g(x)| + \limsup_{t \rightarrow 0^{+}} |e^{it\Delta}g(x) - g(x)|.
\end{align}
By mean value theorem and inequality (\ref{Eq1.4}), we have
\begin{align}\label{Eq2.7n}
| e^{it\Delta}(g)(\gamma(x,t,\theta))-e^{it\Delta}g(x)| \le t\int_{\mathbb{R}^{2}}{|\xi||\hat{g}(\xi)|d\xi},
\end{align}
and
\begin{align}\label{Eq2.8n}
|e^{it\Delta}g(x) -g(x)| \le t\int_{\mathbb{R}^{2}}{|\xi|^{2}|\hat{g}(\xi)|d\xi}.
\end{align}
Inequalities (\ref{Eq2.6n}) - (\ref{Eq2.8n}) imply (\ref{Eq2.5n}).

By (\ref{Eq2.4n}) and (\ref{Eq2.5n}) we have
\begin{align}
&\biggl|\biggl\{{ x \in B(x_{0},R):  \limsup_{\substack{(y,t) \rightarrow (x,0) \\ y\in \Gamma_x}} |e^{it\Delta}(f)(y)-f(x)|} > \lambda\biggl\}\biggl| \nonumber\\
&\le \biggl|\biggl\{{ x \in B(x_{0},R):  \limsup_{\substack{(y,t) \rightarrow (x,0) \\ y\in \Gamma_x}} |e^{it\Delta}(f-g)(y)|} > \frac{ \lambda}{2}\biggl\}\biggl| \nonumber\\
&\quad\quad +\biggl|\biggl\{{ x \in B(x_{0},R):  |f(x)-g(x)|} > \frac{ \lambda}{2}\biggl\}\biggl| \nonumber\\
&\le \biggl|\biggl\{ x \in B(x_{0},R): \mathop{sup}_{(t,\theta) \in (0,1) \times \Theta} |e^{it\Delta}(f-g)(\gamma(x,t,\theta))| > \frac{\lambda}{2}\biggl\}\biggl|  \nonumber\\
&\quad\quad +\biggl|\biggl\{{ x \in B(x_{0},R):  |f(x)-g(x)|} > \frac{ \lambda}{2}\biggl\}\biggl| \nonumber\\
&\lesssim \epsilon + \frac{2^{2}}{\lambda^{2}}\|f-g\|_{H^{s}(\mathbb{R}^{2})}^{2} \nonumber\\
&\le \epsilon +\frac{\epsilon^{\frac{2}{3}}}{C^{2}} \nonumber\\
&\le \epsilon +\epsilon^{\frac{2}{3}},
\end{align}
since we can always assume that $C \ge 1$, which  implies convergence for $f \in H^{s}(\mathbb{R}^{2})$ and almost every $x \in B(x_{0}, R)$. By the arbitrariness of $B(x_{0}, R)$, in fact we can get convergence for almost every $x \in \mathbb{R}^{2}$. This completes the proof of Theorem \ref{theorem1.2}.

\section{Proof of the main theorems for $n \ge 3$}
\textbf{Proof of Theorem \ref{theorem1.4}.} We briefly explain the proof of Theorem \ref{theorem1.4} since most of the details are similar to the proof of Theorem \ref{theorem1.1}. As in the proof of Theorem \ref{theorem1.1}, we only need to prove that for $f$, supp$\hat{f} \subset \{\xi \in \mathbb{R}^{n}: |\xi| \sim \lambda\}$, $\lambda \gg 1$,
 \begin{equation}\label{Eq3.1}
\biggl\|\mathop{sup}_{(t,\theta) \in (0,1) \times \Theta} |e^{it\Delta}f(\gamma(x,t,\theta))|\biggl\|_{L^{p}(B(x_{0},R))} \le C\lambda^{s_{0}+\epsilon} \|f\|_{L^{2}}, \:\ \forall \epsilon >0,
\end{equation}
where $s_{0}=\frac{\beta(\Theta)}{p} + \frac{n}{2(n+1)}$.

 We decompose $\Theta$ into subsets $\Theta= \cup_{k}\Theta_{k}$ with bounded overlap, where each $\Theta_{k}$ is contained in a closed ball with  diameter $\lambda^{-1}$. Then we have
\begin{equation*}
1 \le k \le \lambda^{\beta(\Theta) + \epsilon}.
\end{equation*}
As in the proof of Theorem \ref{theorem1.1}, we can show that for each $k$,
 \begin{equation*}
\biggl\|\mathop{sup}_{(t,\theta) \in (0,1) \times \Theta_{k}} |e^{it\Delta}f(\gamma(x,t,\theta))|\biggl\|_{L^{p}(B(x_{0},R))} \le C\lambda^{\frac{n}{2(n+1)}+\frac{(p-1)\epsilon}{p}} \|f\|_{L^{2}}.
\end{equation*}
Then  we have
\begin{align}
\biggl\|\mathop{sup}_{(t,\theta) \in (0,1) \times \Theta} |e^{it\Delta}f(\gamma(x,t,\theta))|\biggl\|_{L^{p}(B(x_{0},R))}
&\le \biggl(\sum_{k} \biggl\|\mathop{sup}_{(t,\theta) \in (0,1) \times \Theta_{k}} |e^{it\Delta}f(\gamma(x,t,\theta))|\biggl\|_{L^{p}(B(x_{0},R))}^{p}  \biggl)^{1/p} \nonumber\\
&\le C\biggl(\sum_{k} \lambda^{\frac{np}{2(n+1)}+(p-1)\epsilon} \|f\|_{L^{2}}^{p} \biggl)^{1/p} \nonumber\\
&\le  C\lambda^{\frac{\beta{(\Theta)}}{p} +\frac{n}{2(n+1)} +\epsilon} \|f\|_{L^{2}}, \nonumber
\end{align}
which implies inequality (\ref{Eq3.1}).

\textbf{Proof of Theorem \ref{theorem1.5}.} Taking $\gamma (x,t,\theta) = x + t\theta$, $\Theta$ is  the  interior of the unit ball in $\mathbb{R}^{n}$. Then we have
\begin{equation}\label{Eq3.2}
\beta(\Theta) = n,
\end{equation}
and the approach region
\begin{equation}\label{Eq3.3}
\Gamma_{x} = \{\gamma(x,t,\theta):t\in [0,1], \theta\in \Theta\} =\{y: |y-x| < t, t\in [0,1]\}.
\end{equation}
Assuming that (\ref{Eq1.8}) holds true, then it follows from Theorem \ref{theorem1.4} and inequality (\ref{Eq3.3}) that for any $s > \frac{\beta(\Theta)}{p} + \frac{n}{2(n+1)}$,
 \begin{equation}
 \lim_{\substack{(y,t) \rightarrow (x,0) \\ |x-y| < t, t >0}} e^{it\Delta}f(y) = f(x) \hspace{0.2cm} a.e.
 \end{equation}
whenever $f\in H^s(\mathbb{R}^2)$. But according to Theorem 3 in \cite{SS} by Sj\"{o}lin-Sj\"{o}gren, this result fails for any $s \le \frac{n}{2}$. Therefore, we get
\begin{equation}\label{Eq3.5}
 \frac{\beta(\Theta)}{p} + \frac{n}{2(n+1)} \ge \frac{n}{2}.
\end{equation}
Then inequality (\ref{Eq3.5}) and inequality (\ref{Eq3.2}) imply Theorem \ref{theorem1.5}.


\begin{flushleft}
\vspace{0.3cm}\textsc{Wenjuan Li\\School of Mathematics and Statistics\\Northwest Polytechnical University\\710129\\Xi'an, People's Republic of China}

\vspace{0.3cm}\textsc{Huiju Wang\\School of Mathematics Sciences\\University of Chinese Academy of Sciences\\100049\\Beijing, People's Republic of China}

\vspace{0.3cm}\textsc{Dunyan Yan\\School of Mathematics Sciences\\University of Chinese Academy of Sciences\\100049\\Beijing, People's Republic of China}

\end{flushleft}


\begin{thebibliography}{HD}



\normalsize
\baselineskip=17pt

\bibitem{B1} J. Bourgain. Some new estimates on oscillatory integrals. In Essays on Fourier Analysis in Honor of Elias M. Stein (Princeton, NJ, 1991), Princeton Math. Ser. 42, Princeton Univ. Press, Princeton, NJ, 1995: 83-112.

\bibitem{B2} J. Bourgain. On the Schr\"{o}dinger maximal function in higher dimension. Proceedings of the Steklov Institute of Mathematics, 2012, 280(1): 53-66.

\bibitem{B} J. Bourgain. A note on the Schr\"{o}dinger maximal function. Journal d'Analyse Math\'{e}matique, 2016, 130: 393-396.


\bibitem{C} L. Carleson. Some analytic problems related to statistical mechanics. Euclidean harmonic analysis. Springer, Berlin, Heidelberg, 1980: 5-45.

\bibitem{Carbery} A. Carbery. Radial Fourier multipliers and associated maximal functions, in ''Recent
Progress in Fourier Analysis''(I. Peral and J. L. Rubio de Francia, Eds.),
North Holland, Amsterdam, 1985: 49-56.

\bibitem{CLV} C. Cho, S. Lee, A. Vargas. Problems on pointwise convergence of solutions to the Schr\"{o}dinger equation. Journal of Fourier Analysis and Applications, 2012, 18(5): 972-994.

\bibitem{DN} Y. Ding, Y. Niu. Weighted maximal estimates along curve associated with dispersive equations. Analysis and Applications, 2017, 15(2): 225-240.

\bibitem{DK} B. E. J. Dahlberg, C. E. Kenig. A note on the almost everywhere behavior of solutions to the Schr\"{o}dinger equation, in Harmonic Analysis (Minneapolis, Minn., 1981), Lecture Notes in Math. 908, Springer-Verlag, New York, 1982: 205-209.


\bibitem{DGL} X. Du, L. Guth, X. Li. A sharp Schr\"{o}dinger maximal estimate in $\mathbb{R}^{2}$. Annals of Mathematics, 2017, 186: 607-640.

\bibitem{DKWZ} X. Du, J. Kim, H. Wang, R. Zhang. Lower bounds for estimates of  Schr\"{o}dinger maximal function.arxiv Preprint, arxiv: 1902.01430v1, 2019.


\bibitem{DZ} X. Du, R. Zhang. Sharp $L^{2}$ estimates of the Schr\"{o}dinger maximal function in higher dimensions. Annals of Mathematics, 2019, 189: 837-861.


\bibitem{KPV} E. C. Kenig, G. Ponce, L. Vega. Oscillatory integrals and regularity of dispersive equations. Indiana University Mathematics Journal, 1991, 40(1): 33-69.

\bibitem{L} S. Lee. On pointwise convergence of the solutions to Schr\"{o}dinger equations in $\mathbb{R}^{2}$. International Mathematics Research Notices, 2006, 32597: 1-21.

\bibitem{LR} S. Lee, K. M. Rogers. The Schr\"{o}dinger equation along curves and the quantum harmonic oscillator. Advances in Mathematics, 2012, 229: 1359-1379.

\bibitem{LR1} R. Luc\'{a}, K. M. Rogers. An improved necessary condition for the Schr\"{o}dinger maximal estimate. arXiv preprint, arXiv:1506.05325v1, 2015.

\bibitem{Miao} C. Miao, J. Yang, J. Zheng. An improved maximal inequality for 2D fractional order Schr\"{o}dinger operators. Studia Mathematica, 2015, 230: 121-165.

\bibitem{MVV} A. Moyua, A. Vargas, L. Vega. Schr\"{o}dinger maximal function and restricion properties of the Fourier transform. International Mathematics Research Notices, 1996, 16: 793-815.


 \bibitem{Shi} S. Shiraki. Pointwise convergence along restricted directions for the fractional Schr\"{o}dinger equation. arXiv preprint, arXiv:1903.02356v1, 2019.

\bibitem{SS} P. Sj\"{o}gren, P. Sj\"{o}lin. Convergence properties for the time-dependent Schr\"{o}dinger equation. Annales Academire Scientiarurn Fennicre, Series A. I. Mathematica, 1987, 14: 13-25.

\bibitem{S} P. Sj\"{o}lin. Regularity of solutions to the Schr\"{o}dinger equation. Duke Mathematical Journal, 1987, 55(3): 699-715.


\bibitem{V} L. Vega. Schr\"{o}dinger equations: pointwise convergence to the initial data. Proceedings of the American Mathematical Society, 1988, 102(4): 874-878.



\end{thebibliography}
\end{document}